\documentclass[11pt,leqno]{amsart}%
\usepackage{amsmath}
\usepackage{amsfonts}
\usepackage{amssymb}
\usepackage{graphicx}%
\providecommand{\U}[1]{\protect\rule{.1in}{.1in}}
\newtheorem{theorem}{Theorem}

\newtheorem{lemma}[theorem]{Lemma}

\newtheorem{remark}[theorem]{Remark}

\newcommand{\dive}{\operatorname{div}}

\newcommand{\hc}{\hat{c}}

\begin{document}

\title{On $p$-harmonic maps and convex functions}
\date{\today}

\author{Giona Veronelli}
\address{Dipartimento di Matematica\\
Universit\`a di Milano\\
via Saldini 50\\
I-20133 Milano, ITALY}
\email{giona.veronelli@unimi.it}


\begin{abstract}
We prove that, in general, given a $p$-harmonic map $F:M\rightarrow N$ and a convex function $H:N\rightarrow\mathbb{R}$, the composition $H\circ F$ is not $p$-subharmonic. This answers in the negative an open question arisen from a paper by Lin and Wei. By assuming some rotational symmetry on manifolds and functions, we reduce the problem to an ordinary differential inequality. The key of the proof is an asymptotic estimate for the $p$-harmonic map under suitable assumptions on the manifolds. 
\end{abstract}

\maketitle

\section{Introduction and statement of the main result}

A twice differentiable map $F:M\rightarrow N$ between Riemannian manifolds is said to be
$p$-harmonic, $p>1$, if it is a solution of the system
\[
\tau_p(F):=\dive(|dF|^{p-2}dF)=0.
\]
The vector field  $\tau_p(F)$ along $F$ is named the $p$-tension field of $F$ and, whenever $N=\mathbb{R}$,
it is denoted by $\Delta_p$ and called the $p$-laplacian of $F$. In the special situation $p=2$, the $2$-tension field
is traditionally denoted by $\tau(F)$ and the $2$-laplacian reduces to the ordinary Laplace-Beltrami operator $\Delta$ of
the underlying manifold. Moreover, a $2$-harmonic map is simply called a harmonic map.

It is well known that, given a harmonic map between Riemannian
manifolds $F:M\rightarrow N$ and a convex function
$H:N\rightarrow\mathbb{R}$, the composition $H\circ
F:M\rightarrow\mathbb{R}$ is a subharmonic function, namely $
\Delta(H\circ F) \geq 0$. As a matter of fact this property can be
used to characterize the harmonicity of $F$; \cite{Ishihara}. This
is extremely useful since, for example, Liouville type theorems
for harmonic maps into targets supporting a convex function can be
obtained directly from results in linear potential theory of real
valued functions. Such Liouville conclusions, in turn, have
topological consequences e.g. on the homotopy class of maps with
finite energy from a geodesically complete domain; see \cite{SY}
and references therein. It is also known, \cite{Wei-conference},
that $p$-harmonic maps are the natural candidates for the
extension of the above mentioned topological results to maps with
finite higher energies; see also \cite{Wei-indiana} for further
topological aspects of $p$-harmonic maps. In this respect, one is
led to inquire whether the composition of a $p$-harmonic map with
a convex function is  $p$-subharmonic and, therefore, if the
non-linear potential theory of real-valued functions suffices to
get the desired conclusions. By the way, this problem was pointed out by Lin and Wei among a list of open questions in geometry; see Problem 7 in \cite{WL}. It is folklore that, in general, this
is not true, so that one is forced to follow different paths to obtain the results alluded to above;
see e.g. \cite{CL}, \cite{Kawai}, \cite{PRS}. However, to the best
of our knowledge, counterexamples are not yet available in the
literature. The present paper aims to fill this lack.

\medskip

From now on, we let $M_g$ and $N_j$ be $(n+1)$-dimensional Riemannian manifolds with rotationally symmetric metrics
defined as
\begin{align*}
&M_g=([0,+\infty)\times\mathbb{S}^n, ds^2+g^2(s)d\theta^2)\\
&N_j=([0,+\infty)\times\mathbb{S}^n, dt^2+j^2(t)d\theta^2),
\end{align*}
where $g,j\in C^2([0,+\infty))$ satisfy
\begin{equation}\label{condition_gj}
g(0)=j(0)=0,\qquad g'(0)=j'(0)=1,\qquad g(s),j(t)>0\textrm{ for }s,t>0,
\end{equation}
and $(\mathbb{S}^n,d\theta^2)$ is the Euclidean $n$-sphere with its standard metric.
We say that the $C^2$ map $F:M_g\rightarrow N_j$ is rotationally symmetric if
\[
F(s,\theta)=(f(s),\theta)\qquad \forall s>0, \theta\in\mathbb{S}^n,
\]
for some function $f\in C^2([0,\infty))$. Similarly, by a $C^2$
rotationally symmetric real valued function  on $N_j$ we mean a
function $H:N_j\rightarrow\mathbb{R}$ of the form
\[
H(t,\theta)=h(t)\qquad \forall t>0, \theta\in\mathbb{S}^n,
\]
for some $h\in C^2([0,\infty))$.\\

We shall prove the following

\begin{theorem}\label{th_main}
Consider two rotationally symmetric $(n+1)$-dimensional manifolds $M_g$, $N_j$.
Suppose that $(n+1)>p>\max\left\{2,n\right\}$ and assume that the warping functions $g,j\in C^2([0,+\infty))$ have the form
\[
g(s)=(s+\delta^{-\frac{1}{\delta-1}})^{\delta}-\delta^{-\frac{\delta}{\delta-1}},
\qquad j(t)=(t+\sigma^{\frac{1}{1-\sigma}})^{\sigma}-\sigma^{\frac{\sigma}{1-\sigma}},
\]
where $\delta>(p-n)^{-1}>1$ and $0<\sigma<1$. Then, there exist a
$C^2$ rotationally symmetric $p$-harmonic map $F:M_g\rightarrow
N_j$ and a sequence
$\left\{s_k\right\}_{k=1}^{\infty}\rightarrow+\infty$, such that
\[\Delta_p(H\circ F)(s_k,\theta)<0\]
for every rotationally symmetric convex function $H:N_j\rightarrow
\mathbb{R}$, provided the corresponding $h\in C^2([0,+\infty))$
satisfies $h'(t)>0$ for $t>0$.
\end{theorem}

It should be noted that, in the paper \cite{Ishihara} cited above, the author also considers
a special category of harmonic maps,
the harmonic morphisms, which pull back germs of harmonic functions on the target to harmonic functions in the domain.
It is proved that harmonic morphisms are characterized by a weakly horizontal conformality condition.
Recently, \cite{Loubeau}, such a characterization has been extended to the $p$-harmonic setting, $p>2$.
It turns out that the $p$-tension field of the composition of a $p$-harmonic morphism with a generic function
enjoys a very special decomposition. Moreover it's proven that $C^2$ convex functions are $p$-subharmonic; see \cite{WLW-Glob} and \cite{WLW}. Accordingly one has that $p$-harmonic morphisms pull back $p$-subharmonic
functions (and hence pull-back convex functions) to $p$-subharmonic functions.
Such a special decomposition, however, fails to be true in general
for a $p$-harmonic map, and the rotationally symmetric realm provides concrete examples.

\section {Preliminary results}

The proof of Theorem \ref{th_main} relies on a number of preliminary facts on rotationally symmetric $p$-harmonic maps,
ranging from explicit formulas up to existence results and companion asymptotic estimates. In all that follows, notations
are those introduced in Theorem \ref{th_main}.

\subsubsection*{Some fundamental formulas}

The $p$-tension field of the map $F$, on the subset of $M_g$ where $|dF| \neq0$, writes as
\begin{align}\label{p-tension}
\tau_p(F)&=\dive(\left|dF\right|^{p-2}dF)\\
&=\left|dF\right|^{p-2}\left\{\tau(F)+i(\nabla (\lg |dF|^{p-2}))dF\right\}\nonumber,
\end{align}
where $i$ denotes the interior product on $1$-forms. Using the rotational symmetry condition we have
\[
\left|dF\right|^2(s)=\left\{(f'(s))^2+n\frac{j^2(f(s))}{g^2(s)}\right\}.
\]
Furthermore, the tension field of $F$ takes the expression
\begin{equation}\label{tau-F}
\tau(F)=\left\{f''(s)+\frac{n}{g^{2}(s)}\left[g(s)g'(s)f'(s)-j(f(s))j'(f(s))\right]\right\}\left.\frac{\partial}{\partial t}\right|_{f(s)},
\end{equation}
Combining this latter with (\ref{p-tension}), therefore gives
\begin{align}\label{p-f}
\tau_p(F) &=\left|dF\right|^{p-2}(s)\left\{\left[f''(s)+\frac{n}{g^{2}(s)}\left(g(s)g'(s)f'(s)-j(f(s))j'(f(s))\right)\right]\right.\\&+\left.(p-2)\left|dF\right|^{-2}(s)f'(s)\bigg[f'(s)f''(s)\bigg.\right.\nonumber\\
&\left.\left.+n\frac{j(f(s))}{g^3(s)}\left(j'(f(s))f'(s)g(s)-j(f(s))g'(s)\right)\right] \right\}\left.\frac{\partial}{\partial t}\right|_{f(s)}=0,\nonumber
\end{align}
provided $F$ is $p$-harmonic.
Now, we want to compute the $p$-laplacian of the composition $H\circ F$. Using (\ref{p-tension}) with $F$ replaced by $H\circ F$, and setting
\[K(s)=|d(H\circ F)|^{p-2}=|h'(f(s))f'(s)|^{p-2}\]
we conclude
\begin{align}\label{p-HF}
\Delta_p(H\circ F)&= K(s)\left\{h'(f(s))f''(s)+h''(f(s))(f'(s))^2\right.\\
&\left.+ng^{-1}(s)g'(s)f'(s)h'(f(s))\right.\nonumber\\
&+\left.(p-2)\left[h'(f(s))f''(s)+h''(f(s))(f'(s))^2\right]\right\},\nonumber
\end{align}
on the subset

\[M_{+}=\left\{(s,\theta):h'(f(s))f'(s)>0\right\}\subseteq M_g.\]
Finally, we recall that
\[
\mathrm{Hess}\left(  H\right)  \left(  t,\theta\right)
=h^{\prime\prime}\left( t\right)  dt^2+j^{\prime}\left(
t\right)  j\left(  t\right) h^{\prime}\left(  t\right)
d\theta^{2}.
\]
Since the function $j(t)$ defined in Theorem \ref{th_main} is positive and strictly increasing, the
above expression gives us that the convexity of $H$ is equivalent to the set of conditions

\begin{equation}\label{H-conv}
\left\{
\begin{array}{ll}
h''(t)\geq 0 \\
h'(t)\geq 0,
\end{array}
\qquad
\forall t>0.
\right.
\end{equation}

\subsubsection*{Existence results and asymptotic estimates}
The existence of rotationally symmetric $p$-harmonic maps has been investigated
by several authors. Here, we recall the following theorem which encloses in a single statement
Lemma 2.5, Theorem 2.11, Proposition 3.1 and Theorem 3.2 in \cite{CL}(see also Corollary 3.22 in \cite{Leung}).

\begin{theorem}\label{CL}
Suppose that $p>2$ and assume that there exist constants $a>0$ and
$\delta>1$ with $n\delta>p-1$ such that $g,j\in C^2(0,\infty)$,
\[
j(t)>0,\qquad 0\leq j'(t)\leq a\qquad\forall t>0,
\]
and
\[
g(s)\asymp s^{\delta},\qquad g'(s)>0\qquad\textrm{for large }s,
\]
where $g$ and $j$ satisfy the conditions in (\ref{condition_gj}).
Then, for any $\alpha>0$, there is a bounded solution $f\in
C^2[0,+\infty)$ to equation (\ref{p-f}) such that $f(0)=0$,
$f'(0)=\alpha$ and $f(s),f'(s)>0$ for all $s>0$.
\end{theorem}

\begin{remark}
Note that Theorem \ref{CL} and the assumption $h'(t)>0$ imply that $(s,\theta)\in M_+$ and $K(s)\neq0$
for every $s>0$.
\end{remark}

We now want to obtain an asymptotic estimate for $f'(s)$. The following lemma, which is modeled on Corollary 3.13 in \cite{CL}, will play a crucial role.
\begin{lemma}\label{pr_asymp}
Suppose that $(n+1)>p>\max\left\{2,n\right\}$ and assume that there exist constants $a>0$ and $\delta>(p-n)^{-1}>1$, such that $g,j\in C^1(0,\infty)$,
\[
j(t)>0,\qquad 0< j'(t)\leq a\qquad\forall t>0,
\]
and
\[
g(s)\sim C_1s^{\delta},\qquad g'(s)>0\qquad\textrm{for large }s,\ C_1>0,
\]
where $g$ and $j$ satisfy the conditions in (\ref{condition_gj}). Then all positive solutions to equation (\ref{p-f})
satisfy
\begin{equation}\label{asym-f'}
f'(s)\sim Ds^{-\delta\left(n-(p-2)\right)},\qquad\textrm{as}\ s\rightarrow+\infty,
\end{equation}
for some positive constant $D$.
\end{lemma}

\begin{proof}
Let us begin by recalling the following estimate which will be useful later (see (3.7) in \cite{CL})
\begin{equation}\label{3.7}
g^n(s)|dF|^{p-2}(s)f'(s)\leq \tilde{C}\left(\int_{s_0}^{s}r^{(n-2)\delta}|dF|^{p-2}(r)f(r)dr+1\right),
\end{equation}
for $s\geq s_0$, where $s_0$ is some positive constant. From equations (\ref{p-tension}) and (\ref{tau-F}), we get
\[
f'(s)(|dF|^{p-2}(s))'=|dF|^{p-2}(s)\left[\frac{nj(f(s))j'(f(s))}{g^2(s)}-f''(s)-\frac{ng'(s)f'(s)}{g(s)}\right],
\]
from which we obtain that
\[
(g^n|dF|^{p-2}f')'(s)=n|dF|^{p-2}(s)g^{n-2}(s)j(f(s))j'(f(s))\geq0,\qquad \forall s>0.
\]
Hence $(g^n|dF|^{p-2}f')$ is non-decreasing and the following limit holds
\[
(g^n|dF|^{p-2}f')(s)\rightarrow P\in\left(0,+\infty\right],\qquad\textrm{for }s\rightarrow+\infty.
\]
We claim that the limit $P$ is finite. By contradiction suppose $P=+\infty$,
then there exists a sequence $\left\{S_N\right\}_{N=1}^{\infty}$ such that
\[
S_N\rightarrow+\infty\qquad\textrm{and}\qquad(g^n|dF|^{p-2}f')(S_N)=N,
\]
which implies, for all $s\leq S_N$,
\[
g^n(s)(f'(s))^{p-1}\leq g^n(s)|dF|^{p-2}(s)f'(s)\leq N
\]
and
\[
f'(s)\leq N^{\frac{1}{p-1}}g^{-\frac{n}{p-1}}(s)\leq CN^{\frac{1}{p-1}}s^{-\frac{n\delta}{p-1}}.
\]
Moreover, since $p<n+1$ and $\delta>(p-n)^{-1}$ imply $n\delta>(p-1)$,
we can apply Theorem \ref{CL} to deduce that $f'(s)>0$
and $f(s)$ is bounded. Thus
\begin{equation}\label{limit-f}
f(s)\rightarrow\hat{c}>0\ \textrm{ as }s\rightarrow+\infty,
\end{equation}
$f(s)<\hc$ for all $s$ and $f(s)>\hc/2$ for $s$ large enough. Now,
\begin{align}\label{dF-N}
|dF|^2(s)&\leq CN^{\frac{2}{p-1}}s^{-\frac{2n\delta}{p-1}}+n\frac{a^2f^2(s)}{g^2(s)}\\
&\leq C\max\left\{N^{\frac{2}{p-1}}s^{-\frac{2n\delta}{p-1}};s^{-2\delta}\right\}\leq CN^{\frac{2}{p-1}}s^{-2\delta},
\nonumber
\end{align}
since $n>(p-1)$. Hence, from (\ref{3.7}), (\ref{limit-f}) and (\ref{dF-N}), we get
\begin{align*}
N&=(g^n|dF|^{p-2}f')(S_N)\leq \tilde{C}\left(\int_{s_0}^{S_N}r^{(n-2)\delta}|dF|^{p-2}(r)f(r)dr + 1\right)\\
& \leq C\left(N^{\frac{p-2}{p-1}}\int_{s_0}^{S_N}r^{-\delta(p-n)}dr + 1\right)=o(N),\qquad \textrm{as }s\rightarrow+\infty,
\end{align*}
since $p>n$ and $\delta(p-n)>1$. Contradiction. Then
\begin{equation}\label{limit-f'}
f'(s)\sim P|dF|^{2-p}(s)g^{-n}(s),
\end{equation}
for some positive constant $P<\infty$.\\
Now, we need an asymptotic estimate for $|dF|$. Note that
\[
0 \leq \frac{(f'(s))^2g^2(s)}{nj^2(f(s))} \leq \frac{Cs^{-\frac{2n\delta}{p-1}}s^{2\delta}}
{nj^2(\frac{\hc}{2})} = Cs^{2\delta(1-\frac{n}{p-1})} \rightarrow 0,\qquad\textrm{as }s\rightarrow+\infty,
\]
since $p-1<n$. Therefore
\begin{align}\label{limit-dF}
\lim_{s\rightarrow+\infty}\frac{|dF|^2(s)}{n\frac{j^2(f(s))}{g^2(s)}}
=\lim_{s\rightarrow+\infty}\frac{(f'(s))^2g^2(s)}{nj^2(f(s))}+1 = 1,
\end{align}
proving that
\[
|dF|^2(s)\sim n\frac{j^2(f(s))}{g^2(s)}\sim\frac{nj^2(\hc)}{C_1^2}s^{-2\delta}, \text{ as }s\rightarrow+\infty.
\]
Using this information into (\ref{limit-f'}) we conclude
\[
f'(s)\sim Ds^{-\delta(n-(p-2))},\qquad\textrm{with }D:=PC_1^{-n}\left(\frac{C_1^2}{nj^2(\hc)}\right)^{\frac{p-2}{2}}>0,
\]
where $n>p-1>p-2$.
\end{proof}

\section{Proof of Theorem \ref{th_main}}

Observe that the warping functions $g$ and $j$ defined as in Theorem \ref{th_main}
satisfy the assumptions of Theorem \ref{CL} and Lemma \ref{pr_asymp}.
Then, there exists a rotationally symmetric $p$-harmonic map $F(s,\theta)=(f(s),\theta):M_g\rightarrow N_j$
where $f(s)$ is a positive, bounded, increasing function which satisfies (\ref{p-f})
and the asymptotic estimates (\ref{asym-f'}) and (\ref{limit-f}).

Now, multiplying (\ref{p-f}) by $h'(f(s))$, we get
\begin{align*}
&h'(f(s))f''(s)+ng^{-1}(s)h'(f(s))g'(s)f'(s)\\
&=ng^{-2}(s)h'(f(s))j(f(s))j'(f(s))-(p-2)\left|dF\right|^{-2}\left[h'(f(s))f''(s)(f'(s))^2\right.\\
&\left.+ng^{-2}(s)j(f(s))j'(f(s))h'(f(s))(f'(s))^2 - ng^{-3}(s)j^2(f(s))g'(s)h'(f(s))f'(s)\right].
\end{align*}
and inserting the latter into (\ref{p-HF}) we obtain
\begin{align}\label{F+HF}
\Delta_p(H\circ F)= K(s)\tilde{K}(s)\left\{A_1(s)+A_2(s)+A_3(s)\right\},
\end{align}
where we have set
\begin{align*}
&K(s)=|h'(f(s))f'(s)|^{p-2}>0,\qquad\forall s>0;\\
&\tilde{K}(s):=\frac{nj(f(s))h'(f(s))}{|dF|^2(s)g^2(s)}>0,\qquad\forall s>0; \\
&A_1(s):=j'(f(s))\left[(3-p)(f'(s))^2+n\frac{j^2(f(s))}{g^2(s)}\right]; \\
&A_2(s):=(p-2)j(f(s))\left[\frac{g'(s)f'(s)}{g(s)}+f''(s)\right];\\
&A_3(s):=(p-1)(f'(s))^2h''(f(s))\frac{|dF|^2(s)g^2(s)}{nj(f(s))h'(f(s))}.
\end{align*}

\begin{remark}
In the harmonic case $p=2$, (\ref{F+HF}) reduces to
\[
\Delta(H\circ F)=(f'(s))^2h''(f(s))+\frac{n}{g^2(s)}j(f(s))j'(f(s))h'(f(s))
\]
which is always nonegative when $H$ is convex, as we observed in the Introduction.
\end{remark}

Reasoning as in the proof of (\ref{limit-dF}) above, we obtain
\[
A_1(s)\sim nj'(\hc)j^2(\hc)s^{-2\delta},
\]
and
\[
A_3(s)\sim \frac{(p-1)D^2h''(\hc)}{h'(\hc)}j(\hc)s^{-2\delta(n-(p-2))},
\]
as $s\rightarrow+\infty$.
Moreover, according to l'H\^opital rule we have
\[
1=\limsup_{s\rightarrow+\infty}\frac{f'(s)}{Ds^{-\delta(n-(p-2))}}
\leq \limsup_{s\rightarrow+\infty}\frac{f''(s)}{-\delta(n-(p-2))Ds^{-\delta(n-(p-2))-1}}.
\]
Thus, for every $\epsilon>0$ there exists a sequence
$\left\{s_k\right\}_{k=1}^{\infty}$ such that $s_k\rightarrow+\infty$ and
\[
f''(s_k)\leq-\delta(n-(p-2))Ds_k^{-\delta(n-(p-2))-1}(1-\epsilon).
\]
Since
\[
\frac{g'(s)f'(s)}{g(s)}\sim \delta Ds^{-\delta(n-(p-2))-1},\qquad\textrm{as }s\rightarrow+\infty,
\]
we have
\begin{align*}
A_2(s_k) &\leq (p-2)j(\hc)\left\{(1+\epsilon)\delta Ds_k^{-\delta(n-(p-2))-1} \right. \\
&\left.-(1-\epsilon)\delta(n-(p-2))Ds_k^{-\delta(n-(p-2))-1}\right\}.
\end{align*}
for $k$ large enough. Now recall that, by the assumptions on $n$ and $p$, it holds
\[D\delta(1-(n-(p-2)))<0.\]
Therefore, we can choose
\[0<\epsilon<(n+1-p)/(n+3-p)\]
in order to ensure that, for every $k$ large enough,
\[A_2(s_k)<0.\]
Finally note that, as $s_k\rightarrow+\infty$,
$A_1(s_k)$ and $A_3(s_k)$ decay faster than $A_2(s_k)$ because, by the
assumptions on $\delta,n$ and $p$,
\[
-2\delta(n-(p-2))<-1-\delta(n-(p-2)),\qquad-2\delta<-1-\delta(n-(p-2)).
\]
According to (\ref{F+HF}), this shows that, for $k$ large enough, $\Delta_p(H\circ F)(s_k) < 0$, as requested.

\subsection*{Acknowledgement}
The author is deeply grateful to his advisor Stefano Pigola for his guidance and constant encouragement during the preparation of the manuscript.\\
Moreover he would like to thank Professor Shihshu Walter Wei for pointing me out some useful references and for some suggestions that helped to improve the presentation of the paper.

\end{document}